\documentclass{article}
\usepackage{amssymb}
\usepackage{amsmath}
\usepackage{amsfonts}
\usepackage{chbibref}

\newtheorem{theorem}{Theorem}

\newtheorem{definition}[theorem]{Definition}

\newtheorem{notation}[theorem]{Notation}

\newenvironment{proof}[1][Proof]{\noindent\textbf{#1.} }{\ \rule{0.5em}{0.5em}}
\setbibref{ }
\input{tcilatex}

\begin{document}

\title{Homogenization of monotone parabolic problems with an arbitrary
number of spatial and temporal scales\noindent \bigskip \noindent }
\author{T. Danielsson, L. Flod\'{e}n, P. Johnsen, M. Olsson Lindberg \\
tatiana.danielsson@miun.se, liselott.floden@miun.se,\\
pernilla.johnsen@miun.se, marianne.olssonlindberg@miun.se\\
Department of Mathematics and Science Education,\\
Mid Sweden University, S-83125 \"{O}stersund, Sweden}
\date{}
\maketitle

\begin{abstract}
\vspace{-0.1cm}In this paper we prove a general homogenization result for
monotone parabolic problems with an arbitrary number of microscopic scales
in space as well as in time, where the scale functions are not necessarily
powers of epsilon. The main tools for the homogenization procedure are
multiscale convergence and very weak multiscale convergence, both adapted to
evolution problems. At the end of the paper an example is given to
concretize the use of the main result.
\end{abstract}

\section{Introduction}

The mathematical theory of nonlinear partial differential equations plays an
important role in e.g. applied mathematics and physics. In this paper we
present a homogenization result for the general monotone parabolic problem
with multiple spatial and temporal scales%
\begin{eqnarray}
\partial _{t}u^{\varepsilon }\left( x,t\right) -\nabla \cdot a\left( \frac{x%
}{\hat{\varepsilon}_{1}},\ldots ,\frac{x}{\hat{\varepsilon}_{n}},\frac{t}{%
\check{\varepsilon}_{1}},\ldots ,\frac{t}{\check{\varepsilon}_{m}},\nabla
u^{\varepsilon }\left( x,t\right) \right) \!\!\! &=&\!\!\!f\left( x,t\right)
\,\text{\thinspace in\ }\Omega _{T}\text{,}  \notag \\
u^{\varepsilon }\left( x,t\right) \!\!\! &=&\!\!\!0\,\text{\thinspace
\thinspace on \thinspace }\partial \Omega \!\!\times \!\!(0,T)\text{,}%
\,\,\,\,\,{}~~~~  \label{first one} \\
u^{\varepsilon }\left( x,0\right) \!\!\! &=&\!\!\!u^{0}\left( x\right) \,%
\text{\thinspace in\ }\Omega \text{,}  \notag
\end{eqnarray}%
where $f\in L^{2}(\Omega _{T})$ and $u^{0}\in L^{2}(\Omega )$. Here $\Omega $
is an open bounded set in $%
\mathbb{R}
^{N}$ with smooth boundary and $\Omega _{T}=\Omega \times (0,T)$. We let $%
Y=(0,1)^{N}$ and $S=(0,1)$ and we assume that $a$ is $Y$-periodic in the $n$
first variables and $S$-periodic in the following $m$ variables. Finally we
let $\hat{\varepsilon}_{k}$ for $k=1,\ldots ,n$ and $\check{\varepsilon}_{j}$
for $j=1,\ldots ,m$ be scale functions depending on $\varepsilon $ that tend
to zero as $\varepsilon $ does, where the scales are assumed to fulfil
certain conditions of separatedness.

The homogenization of (\ref{first one}) means studying the asymptotic
behavior of the corresponding sequence of solutions $u^{\varepsilon }$ as $%
\varepsilon $ tends to zero and finding the limit problem 
\begin{eqnarray*}
\partial _{t}u\left( x,t\right) -\nabla \cdot b\left( x,t,\nabla u\left(
x,t\right) \right) &=&f\left( x,t\right) \text{ in }\Omega _{T}\text{,} \\
u\left( x,t\right) &=&0\text{ on }\partial \Omega \times (0,T)\text{,} \\
u\left( x,0\right) &=&u^{0}\left( x\right) \text{ in }\Omega \text{,}
\end{eqnarray*}%
which admits the function $u$, the limit of $\left\{ u^{\varepsilon
}\right\} $, as its unique solution. Here $b$ is characterized by local
problems, one for each microscopic spatial scale. For more informative texts
on homogenization theory we suggest e.g. \cite{Al1}, \cite{CoDo} and \cite%
{LNW}.

The main tools to carry out the homogenization process for (\ref{first one})
are multiscale convergence and very weak multiscale convergence in the
evolution setting. Here very weak multiscale convergence, see e.g. \cite%
{FHOP} and \cite{FHOLP}, is the key to handling the difficulties that appear
when rapid time oscillations are present. The nonlinearity of the problem is
treated by applying the perturbed test functions method.

Homogenization results for linear parabolic equations with oscillations in
one spatial scale and one temporal scale were studied by using asymptotic
expansions in \cite{BLP}. In \cite{Ho1} parabolic problems containing fast
oscillations in space as well as in time were treated for the first time
applying two-scale convergence methods. Parabolic homogenization problems
have also been investigated in e.g. \cite{FlOl1} and \cite{FHOLP1}\ for
different choices of fixed scales. Linear parabolic problems with an
arbitrary number of scales in both space and time were homogenized in \cite%
{FHOLP}. Homogenization results for monotone, not necessarily linear,
problems have been presented in e.g. \cite{FlOl}, \cite{NgWo}, \cite{FHOS}, 
\cite{Wo1} and \cite{Wo2}. The case with one spatial microscale and an
arbitrary number of temporal scales was treated by Persson in \cite{Per1}.

The paper is organized in the following way. In Section 2 we give some
preparatory theory concerning multiscale and very weak multiscale
convergence. In Section 3 we present the homogenization result for (\ref%
{first one}) and in the last section we look at a special case of (\ref%
{first one}) to illustrate the use of the presented result.

\begin{notation}
We let $F_{\sharp }(Y)$ be the space of all functions in $F_{loc}(%
\mathbb{R}
^{N})$ which are the periodic repetition of some function in $F(Y)$. We also
let $Y_{k}=Y$ for $k=1,\ldots ,n$, $Y^{n}=Y_{1}\times \cdots \times Y_{n}$
(the nN-dimensional open unit cell), $y^{n}=y_{1},\ldots ,y_{n}$
(corresponding local spatial multivariable), $dy^{n}=dy_{1}\cdots dy_{n}$, $%
S_{j}=S$ for $j=1,\ldots ,m$, $S^{m}=S_{1}\times \cdots \times S_{m}$ (the
m-dimensional open unit cell), $s^{m}=s_{1},\ldots ,s_{m}$ (corresponding
local temporal multivariable), $ds^{m}=ds_{1}\cdots ds_{m}$ and $\mathcal{Y}%
_{n,m}=Y^{n}\times S^{m}$, where we interpret $\mathcal{Y}_{0,m}$ as $S^{m}$%
. We let $\hat{\varepsilon}_{k}\left( \varepsilon \right) $, for $k=1,\ldots
,n$, and $\check{\varepsilon}_{j}\left( \varepsilon \right) $, $j=1,\ldots
,m $, be strictly positive functions such that $\hat{\varepsilon}_{k}\left(
\varepsilon \right) $ and $\check{\varepsilon}_{j}\left( \varepsilon \right) 
$ go to zero when $\varepsilon $ does. We also use the notations $\hat{%
\varepsilon}^{n}=\hat{\varepsilon}_{1},\ldots ,\hat{\varepsilon}_{n}$ and $%
\check{\varepsilon}^{m}=\check{\varepsilon}_{1},\ldots ,\check{\varepsilon}%
_{m}$ and furthermore $\frac{x}{\hat{\varepsilon}^{n}}$ denotes $\frac{x}{%
\hat{\varepsilon}_{1}},\ldots ,\frac{x}{\hat{\varepsilon}_{n}}$ and,
similarly, by $\frac{t}{\check{\varepsilon}^{m}}$ we mean $\frac{t}{\check{%
\varepsilon}_{1}},\ldots ,\frac{t}{\check{\varepsilon}_{m}}$.
\end{notation}

\section{Multiscale and very weak multiscale convergence}

In \cite{Ng1} Nguetseng presented a new homogenization technique based on a
certain type of convergence\ which has become known as two-scale
convergence.\ This was extended in \cite{AlBr} to so-called multiscale
convergence, which allows use of multiple scales and makes it possible to
capture numerous types of spatial microscopic oscillations. Below we define
evolution multiscale convergence, which is the further development of
multiscale convergence to include temporal oscillations, see also \cite%
{FHOLP}.

\begin{definition}
\label{Defintion multiscale}A sequence $\left\{ u^{\varepsilon }\right\} $
in $L^{2}(\Omega _{T})$ is said to $(n+1,m+1)$-scale converge to $u_{0}\in
L^{2}(\Omega _{T}\times \mathcal{Y}_{n,m})$ if 
\begin{gather*}
\int_{\Omega _{T}}u^{\varepsilon }\left( x,t\right) v\left( x,t,\frac{x}{%
\hat{\varepsilon}^{n}},\frac{t}{\check{\varepsilon}^{m}}\right) dxdt \\
\rightarrow \int_{\Omega _{T}}\int_{\mathcal{Y}_{n,m}}u_{0}\left(
x,t,y^{n},s^{m}\right) v\left( x,t,y^{n},s^{m}\right) dy^{n}ds^{m}dxdt
\end{gather*}%
for any $v\in L^{2}\left( \Omega _{T};C_{\sharp }(\mathcal{Y}_{n,m})\right) $%
. We write 
\begin{equation*}
u^{\varepsilon }\left( x,t\right) \overset{n+1,m+1}{\rightharpoonup }%
u_{0}\left( x,t,y^{n},s^{m}\right) \text{.}
\end{equation*}
\end{definition}

Next we define some concepts regarding relations between scale functions.

\begin{definition}
We say that the scales in a list $\left\{ \varepsilon _{1},\ldots
,\varepsilon _{n}\right\} $ are separated if%
\begin{equation*}
\lim_{\varepsilon \rightarrow 0}\frac{\varepsilon _{k+1}}{\varepsilon _{k}}=0
\end{equation*}%
for $k=1,\ldots ,n-1$ and that the scales are well-separated if there exists
a positive integer $l$ such that%
\begin{equation*}
\lim_{\varepsilon \rightarrow 0}\frac{1}{\varepsilon _{k}}\left( \frac{%
\varepsilon _{k+1}}{\varepsilon _{k}}\right) ^{l}=0
\end{equation*}%
for $k=1,\ldots ,n-1$.
\end{definition}

\begin{definition}
\label{Lists of scales}Let $\left\{ \hat{\varepsilon}_{1},\ldots ,\hat{%
\varepsilon}_{n}\right\} $ and $\left\{ \check{\varepsilon}_{1},\ldots ,%
\check{\varepsilon}_{m}\right\} $ be lists of (well-)separated scales.
Collect all elements from both lists in one common list. If from possible
duplicates, where by duplicates we mean scales which tend to zero equally
fast, one member of each pair is removed and the list in order of magnitude
of all the remaining elements is (well-)separated, the lists $\left\{ \hat{%
\varepsilon}_{1},\ldots ,\hat{\varepsilon}_{n}\right\} $ and $\left\{ \check{%
\varepsilon}_{1},\ldots ,\check{\varepsilon}_{m}\right\} $ are said to be
jointly (well-)separated.
\end{definition}

We give the two following theorems\ which state a compactness result for $%
\left( n+1,m+1\right) $-scale convergence and a characterization of
multiscale limits for gradients, respectively.

\begin{theorem}
\label{thmultiskale}Let $\left\{ u^{\varepsilon }\right\} $ be a bounded
sequence in $L^{2}(\Omega _{T})$ and suppose that the lists $\left\{ \hat{%
\varepsilon}_{1},\ldots ,\hat{\varepsilon}_{n}\right\} $ and $\left\{ \check{%
\varepsilon}_{1},\ldots ,\check{\varepsilon}_{m}\right\} $ are jointly
separated. Then there exists a $u_{0}$ in $L^{2}(\Omega _{T}\times \mathcal{Y%
}_{n,m})$ such that, up to a subsequence, 
\begin{equation*}
u^{\varepsilon }\left( x,t\right) \overset{n+1,m+1}{\rightharpoonup }%
u_{0}\left( x,t,y^{n},s^{m}\right) \text{.}
\end{equation*}
\end{theorem}

\begin{proof}
See Theorem 2.66 in \cite{Per2} or Theorem A.1. in \cite{FHOLP}.
\end{proof}

The space $W_{2}^{1}(0,T;H_{0}^{1}(\Omega ),L^{2}(\Omega ))$ that appears in
the theorem below is the space of all functions in $L^{2}(0,T;H_{0}^{1}(%
\Omega ))$ such that the time derivative belongs to $L^{2}(0,T;H^{-1}(\Omega
))$.

\begin{theorem}
\label{gradkarmulti}Let $\left\{ u^{\varepsilon }\right\} $ be a bounded
sequence in $W_{2}^{1}(0,T;H_{0}^{1}(\Omega ),L^{2}(\Omega ))$ and suppose
that the lists $\left\{ \hat{\varepsilon}_{1},\ldots ,\hat{\varepsilon}%
_{n}\right\} $ and $\left\{ \check{\varepsilon}_{1},\ldots ,\check{%
\varepsilon}_{m}\right\} $ are jointly well-separated. Then, up to a
subsequence,%
\begin{eqnarray*}
u^{\varepsilon }\left( x,t\right) &\rightarrow &u\left( x,t\right) \text{ in 
}L^{2}(\Omega _{T})\text{,} \\
u^{\varepsilon }\left( x,t\right) &\rightharpoonup &u\left( x,t\right) \text{
in }L^{2}(0,T;H_{0}^{1}(\Omega ))
\end{eqnarray*}%
and%
\begin{equation*}
\nabla u^{\varepsilon }\left( x,t\right) \overset{n+1,m+1}{\rightharpoonup }%
\nabla u\left( x,t\right) +\sum\limits_{j=1}^{n}\nabla _{y_{j}}u_{j}\left(
x,t,y^{j},s^{m}\right)
\end{equation*}%
where $u\in W_{2}^{1}(0,T;H_{0}^{1}(\Omega ),L^{2}(\Omega ))$ and $u_{j}\in
L^{2}(\Omega _{T}\times \mathcal{Y}_{j-1,m};H_{\sharp }^{1}(Y_{j})/%
\mathbb{R}
)$ for $j=1,\ldots ,n$.
\end{theorem}

\begin{proof}
See Theorem 2.74 in \cite{Per2} or\ Theorem 4 in \cite{FHOLP}.
\end{proof}

Multiscale convergence is very useful for homogenization of problems
involving rapid oscillations on several micro levels. Unfortunately, we can
only use this for sequences bounded in the $L^{2}$-norm and when rapid time
oscillations are present we encounter sequences that do not possess this
boundedness. Multiscale convergence has a large class of test functions and
the limit captures both the global trend and the microscopic oscillations.
If we downsize this class to only capture the microscopic fluctuations it
becomes possible to handle certain sequences that are not required to be
bounded in any Lebesgue space. This is the idea behind so-called very weak
multiscale convergence. A first compactness result of very weak multiscale
convergence type was given in \cite{Ho1}, see also \cite{NgWo}, \cite{FHOP}
and \cite{FHOLP2}.

\begin{definition}
A sequence $\left\{ w^{\varepsilon }\right\} $ in $L^{1}(\Omega _{T})$ is
said to $(n+1,m+1)$-scale converge very weakly to $w_{0}\in L^{1}(\Omega
_{T}\times \mathcal{Y}_{n,m})$ if%
\begin{gather*}
\int_{\Omega _{T}}w^{\varepsilon }\left( x,t\right) v_{1}\left( x,\frac{x}{%
\hat{\varepsilon}_{1}},\ldots ,\frac{x}{\hat{\varepsilon}_{n-1}}\right)
c\left( t,\frac{t}{\check{\varepsilon}_{1}},\ldots ,\frac{t}{\check{%
\varepsilon}_{m}}\right) v_{2}\left( \frac{x}{\hat{\varepsilon}_{n}}\right)
dxdt \\
\rightarrow \int_{\Omega _{T}}\int_{\mathcal{Y}_{n,m}}w_{0}\left(
x,t,y^{n},s^{m}\right) v_{1}(x,y^{n-1})c(t,s^{m})v_{2}(y_{n})dy^{n}ds^{m}dxdt
\end{gather*}%
for any $v_{1}\in D(\Omega ;C_{\sharp }^{\infty }(Y^{n-1})),$ $v_{2}\in
C_{\sharp }^{\infty }\left( Y_{n}\right) /%
\mathbb{R}
$ and $c\in D(0,T;C_{\sharp }^{\infty }\left( S^{m}\right) )$ where 
\begin{equation*}
\int_{Y_{n}}w_{0}\left( x,t,y^{n},s^{m}\right) dy_{n}=0.
\end{equation*}%
We write%
\begin{equation*}
w^{\varepsilon }\left( x,t\right) \underset{vw}{\overset{n+1,m+1}{%
\rightharpoonup }}w_{0}\left( x,t,y^{n},s^{m}\right) \text{.}
\end{equation*}
\end{definition}

The following theorem is essential for the homogenization of (\ref{first one}%
).

\begin{theorem}
\label{T vw}Let $\left\{ u^{\varepsilon }\right\} $ be a bounded sequence in 
$W_{2}^{1}(0,T;H_{0}^{1}(\Omega ),L^{2}(\Omega ))$ and assume that the lists 
$\left\{ \hat{\varepsilon}_{1},\ldots ,\hat{\varepsilon}_{n}\right\} $ and $%
\left\{ \check{\varepsilon}_{1},\ldots ,\check{\varepsilon}_{m}\right\} $
are jointly well-separated. Then there exists a subsequence such that%
\begin{equation*}
\frac{u^{\varepsilon }\left( x,t\right) }{\varepsilon _{n}}\overset{n+1,m+1}{%
\underset{vw}{\rightharpoonup }}u_{n}\left( x,t,y^{n},s^{m}\right) \text{,}
\end{equation*}%
where, for $n=1,2,\ldots $, $u_{n}\in L^{2}(\Omega _{T}\times \mathcal{Y}%
_{n-1,m};H_{\sharp }^{1}(Y_{n})/%
\mathbb{R}
)$ are the same as in Theorem \ref{gradkarmulti}.
\end{theorem}

\begin{proof}
See Theorem 2.78 in \cite{Per2} or Theorem 7\ in \cite{FHOLP}.
\end{proof}

\section{The homogenization result}

We study the homogenization of the problem%
\begin{eqnarray}
\partial _{t}u^{\varepsilon }\left( x,t\right) -\nabla \cdot a\left( \frac{x%
}{\hat{\varepsilon}^{n}},\frac{t}{\check{\varepsilon}^{m}},\nabla
u^{\varepsilon }\left( x,t\right) \right) &=&f\left( x,t\right) \text{ in }%
\Omega _{T}\text{,}  \notag \\
u^{\varepsilon }\left( x,t\right) &=&0\text{ on }\partial \Omega \times (0,T)%
\text{,}  \label{equation1} \\
u^{\varepsilon }\left( x,0\right) &=&u^{0}\left( x\right) \text{ in }\Omega 
\text{,}  \notag
\end{eqnarray}%
where $f\in L^{2}(\Omega _{T})$ and $u^{0}\in L^{2}(\Omega )$. Here we
assume that%
\begin{equation*}
a:%
\mathbb{R}
^{nN}\times 
\mathbb{R}
^{m}\times 
\mathbb{R}
^{N}\rightarrow 
\mathbb{R}
^{N}
\end{equation*}%
satisfies the following structure conditions, where $C_{0}$ and $C_{1}$ are
positive constants and $0<\alpha \leq 1$:

\begin{enumerate}
\item[(i)] $a(y^{n},s^{m},0)=0$ for all $\left( y^{n},s^{m}\right) \in 
\mathbb{R}
^{nN}\times 
\mathbb{R}
^{m}$.

\item[(ii)] $a(\cdot ,\cdot ,\xi )$ is $\mathcal{Y}_{n,m}$-periodic in $%
\left( y^{n},s^{m}\right) $ and continuous for all $\xi \in 
\mathbb{R}
^{N}$.

\item[(iii)] $a(y^{n},s^{m},\cdot )$ is continuous for all $\left(
y^{n},s^{m}\right) \in 
\mathbb{R}
^{nN}\times 
\mathbb{R}
^{m}$.

\item[(iv)] $\left( a\left( y^{n},s^{m},\xi \right) -a\left( y^{n},s^{m},\xi
^{\prime }\right) \right) \cdot (\xi -\xi ^{\prime })\geq C_{0}\left\vert
\xi -\xi ^{\prime }\right\vert ^{2}$ for all
\end{enumerate}

$\ \ \ \left( y^{n},s^{m}\right) \in 
\mathbb{R}
^{nN}\times 
\mathbb{R}
^{m}$ and all $\xi ,\xi ^{\prime }\in 
\mathbb{R}
^{N}$.

\begin{enumerate}
\item[(v)] $\left\vert a\left( y^{n},s^{m},\xi \right) -a\left(
y^{n},s^{m},\xi ^{\prime }\right) \right\vert \leq C_{1}(1+\left\vert \xi
\right\vert +\left\vert \xi ^{\prime }\right\vert )^{1-\alpha }\left\vert
\xi -\xi ^{\prime }\right\vert ^{\alpha }$ for all
\end{enumerate}

$\ \ \ \left( y^{n},s^{m}\right) \in 
\mathbb{R}
^{nN}\times 
\mathbb{R}
^{m}$ and all $\xi ,\xi ^{\prime }\in 
\mathbb{R}
^{N}$.

\bigskip \noindent Under these conditions, problem (\ref{equation1})
possesses a unique solution, see Theorem 30.A (a) in \cite{ZeiIIB}, and the
a priori estimate%
\begin{equation*}
\left\Vert u^{\varepsilon }\right\Vert _{W_{2}^{1}(0,T;H_{0}^{1}(\Omega
),L^{2}(\Omega ))}<C
\end{equation*}%
holds true for some $C>0$, see Proposition 3.16 in\ \cite{Per2}. Finally we
assume that the lists $\left\{ \hat{\varepsilon}^{n}\right\} $ and $\left\{ 
\check{\varepsilon}^{m}\right\} $ in (\ref{equation1}) are jointly
well-separated.

In order to formulate the theorem below in a neat way we define some numbers
determined by how the scale functions present are related to each other. We
define $d_{i}$ and $\rho _{i}$, $i=1,\ldots ,n$, as follows:

\begin{enumerate}
\item[(I)] If 
\begin{equation*}
\underset{\varepsilon \rightarrow 0}{\lim }\frac{\check{\varepsilon}_{1}}{%
\left( \hat{\varepsilon}_{i}\right) ^{2}}=0\text{,}
\end{equation*}%
then $d_{i}=m$. If%
\begin{equation*}
\underset{\varepsilon \rightarrow 0}{\lim }\frac{\check{\varepsilon}_{j}}{%
\left( \hat{\varepsilon}_{i}\right) ^{2}}>0\text{ and }\underset{\varepsilon
\rightarrow 0}{\lim }\frac{\check{\varepsilon}_{j+1}}{\left( \hat{\varepsilon%
}_{i}\right) ^{2}}=0
\end{equation*}%
for some $j=1,\ldots ,m-1$, then $d_{i}=m-j$. If%
\begin{equation*}
\underset{\varepsilon \rightarrow 0}{\lim }\frac{\check{\varepsilon}_{m}}{%
\left( \hat{\varepsilon}_{i}\right) ^{2}}>0\text{,}
\end{equation*}%
then $d_{i}=0$.

\item[(II)] If%
\begin{equation*}
\underset{\varepsilon \rightarrow 0}{\lim }\frac{\left( \hat{\varepsilon}%
_{i}\right) ^{2}}{\check{\varepsilon}_{j}}=C\text{,}
\end{equation*}%
$0<C<\infty $, for some $j=1,\ldots ,m$ we say that we have resonance and we
let $\rho _{i}=C$, otherwise $\rho _{i}=0$.
\end{enumerate}

\noindent This means that $d_{i}$ is the number of temporal scales faster
than the square of the spatial scale in question and $\rho _{i}$ indicates
whether there is resonance or not.

We are now prepared to give and prove the main theorem of the paper. Here $%
W_{2_{{\large \sharp }}}^{1}(S;H_{\sharp }^{1}(Y)/%
\mathbb{R}
,L_{\sharp }^{2}(Y)/%
\mathbb{R}
)$ denotes the space of all functions $u$ such that $u\in L_{\sharp
}^{2}(S;H_{\sharp }^{1}(Y)/%
\mathbb{R}
)$ and $\partial _{s}u\in L_{\sharp }^{2}(S;(H_{\sharp }^{1}(Y)/%
\mathbb{R}
)^{\prime })$.

\begin{theorem}
\label{first inline}Let $\left\{ u^{\varepsilon }\right\} $ be a sequence of
solutions in $W_{2}^{1}(0,T;H_{0}^{1}(\Omega ),L^{2}(\Omega ))$ to (\ref%
{equation1}). Then it holds that%
\begin{equation}
u^{\varepsilon }\left( x,t\right) \rightarrow u\left( x,t\right) \text{ in }%
L^{2}(\Omega _{T})\text{,}  \label{equation 2}
\end{equation}%
\begin{equation}
u^{\varepsilon }\left( x,t\right) \rightharpoonup u\left( x,t\right) \text{
in }L^{2}(0,T;H_{0}^{1}(\Omega ))  \label{equation3}
\end{equation}%
and%
\begin{equation*}
\nabla u^{\varepsilon }\left( x,t\right) \overset{n+1,m+1}{\rightharpoonup }%
\nabla u\left( x,t\right) +\overset{n}{\underset{j=1}{\sum }}\nabla
_{y_{j}}u_{j}\left( x,t,y^{j},s^{m-d_{j}}\right) \text{,}
\end{equation*}%
where $u\in W_{2}^{1}(0,T;H_{0}^{1}(\Omega ),L^{2}(\Omega ))$ is the unique
solution to%
\begin{eqnarray*}
\partial _{t}u\left( x,t\right) -\nabla \cdot b\left( x,t,\nabla u\left(
x,t\right) \right) &=&f\left( x,t\right) \text{ in }\Omega _{T}\text{,} \\
u\left( x,t\right) &=&0\text{ on }\partial \Omega \times (0,T)\text{,} \\
u\left( x,0\right) &=&u^{0}\left( x\right) \text{ in }\Omega
\end{eqnarray*}%
with%
\begin{gather*}
b\left( x,t,\nabla u\left( x,t\right) \right) \\
=\int_{\mathcal{Y}_{n,m}}a\left( y^{n},s^{m},\nabla u\left( x,t\right)
+\dsum\limits_{j=1}^{n}\nabla _{y_{j}}u_{j}\!\left(
x,t,y^{j},s^{m-d_{j}}\right) \right) dy^{n}ds^{m}\text{,}
\end{gather*}%
where $u_{i}\in L^{2}(\Omega _{T}\times \mathcal{Y}_{i-1,m-d_{i}};H_{\sharp
}^{1}(Y_{i})/%
\mathbb{R}
)$ for $i=1,\ldots ,n$. Here $u_{i}$, for $i=1,\ldots ,n$, are the unique
solutions to the system of local problems%
\begin{gather}
\rho _{i}\partial _{s_{m-d_{i}}}u_{i}\left( x,t,y^{i},s^{m-d_{i}}\right) 
\notag \\
-\nabla _{y_{i}}\!\cdot \!\!\int_{S_{m-d_{i}+1}}\!\!\cdots
\!\int_{S_{m}}\!\int_{Y_{i+1}}\!\!\cdots \!\int_{Y_{n}}\!a\!\left(
y^{n},s^{m},\nabla u\left( x,t\right) +\!\dsum\limits_{j=1}^{n}\nabla
_{y_{j}}u_{j}\!\left( x,t,y^{j},s^{m-d_{j}}\right) \!\!\right)
\label{equation5} \\
\times dy_{n}\cdots dy_{i+1}ds_{m}\cdots ds_{m-d_{i}+1}=0  \notag
\end{gather}%
if we assume that $u_{i}\in L^{2}(\Omega _{T}\times \mathcal{Y}%
_{i-1,m-d_{i}-1};W_{2_{{\large \sharp }}}^{1}(S_{m-d_{i}};H_{\sharp
}^{1}(Y_{i})/%
\mathbb{R}
,L_{\sharp }^{2}(Y_{i})/%
\mathbb{R}
))$ when $\rho _{i}\neq 0$.
\end{theorem}

\begin{proof}
The lists $\left\{ \hat{\varepsilon}^{n}\right\} $ and $\left\{ \check{%
\varepsilon}^{m}\right\} $ of scales are jointly well-separated and since $%
\left\{ u^{\varepsilon }\right\} $ is bounded in $W_{2}^{1}(0,T;H_{0}^{1}(%
\Omega ),L^{2}(\Omega ))$ Theorem \ref{gradkarmulti} is applicable and
hence, up to a subsequence,%
\begin{eqnarray*}
u^{\varepsilon }\left( x,t\right) &\rightarrow &u\left( x,t\right) \text{ in 
}L^{2}(\Omega _{T})\text{,} \\
u^{\varepsilon }\left( x,t\right) &\rightharpoonup &u\left( x,t\right) \text{
in }L^{2}(0,T;H_{0}^{1}(\Omega ))
\end{eqnarray*}%
and 
\begin{equation*}
\nabla u^{\varepsilon }\left( x,t\right) \overset{n+1,m+1}{\rightharpoonup }%
\nabla u\left( x,t\right) +\overset{n}{\underset{j=1}{\sum }}\nabla
_{y_{j}}u_{j}(x,t,y^{j},s^{m})\text{,}
\end{equation*}%
where $u\in W_{2}^{1}(0,T;H_{0}^{1}(\Omega ),L^{2}(\Omega ))$ and $u_{j}\in
L^{2}(\Omega _{T}\times \mathcal{Y}_{j-1,m};H_{\sharp }^{1}(Y_{j})/%
\mathbb{R}
)$ for $j=1,\ldots ,n$.

The weak form of (\ref{equation1}) reads: find $u^{\varepsilon }\in
W_{2}^{1}(0,T;H_{0}^{1}(\Omega ),L^{2}(\Omega ))$ such that%
\begin{gather}
\int_{\Omega _{T}}-u^{\varepsilon }\left( x,t\right) v\left( x\right)
\partial _{t}c\left( t\right) +a\left( \frac{x}{\hat{\varepsilon}^{n}},\frac{%
t}{\check{\varepsilon}^{m}},\nabla u^{\varepsilon }\left( x,t\right) \right)
\cdot \nabla v\left( x\right) c\left( t\right) dxdt  \notag \\
=\int_{\Omega _{T}}f\left( x,t\right) v\left( x\right) c\left( t\right) dxdt
\label{svag form}
\end{gather}%
for all $v\in H_{0}^{1}(\Omega )$ and $c\in D(0,T)$. By choosing $\xi
^{\prime }=0$ in (v) we have 
\begin{equation*}
\left\vert a\left( y^{n},s^{m},\xi \right) \right\vert \leqslant C_{1}\left(
1+\left\vert \xi \right\vert \right) ^{1-\alpha }\left\vert \xi \right\vert
^{\alpha }
\end{equation*}%
and since%
\begin{equation*}
C_{1}\left( 1+\left\vert \xi \right\vert \right) ^{1-\alpha }\left\vert \xi
\right\vert ^{\alpha }<C_{1}\left( 1+\left\vert \xi \right\vert \right)
^{1-\alpha }\left( 1+\left\vert \xi \right\vert \right) ^{\alpha }
\end{equation*}%
we obtain%
\begin{equation}
\left\vert a\left( y^{n},s^{m},\xi \right) \right\vert <C_{1}\left(
1+\left\vert \xi \right\vert \right) \text{.}  \label{flyttning}
\end{equation}%
The boundedness of $\left\{ u^{\varepsilon }\right\} $ in $%
L^{2}(0,T;H_{0}^{1}(\Omega ))$\ together with (\ref{flyttning}) gives, up to
a subsequence, that 
\begin{equation*}
a\left( \frac{x}{\hat{\varepsilon}^{n}},\frac{t}{\check{\varepsilon}^{m}}%
,\nabla u^{\varepsilon }\left( x,t\right) \right) \overset{n+1,m+1}{%
\rightharpoonup }a_{0}\left( x,t,y^{n},s^{m}\right)
\end{equation*}%
for some $a_{0}\in L^{2}(\Omega _{T}\times \mathcal{Y}_{n,m})^{N}$ due to
Theorem \ref{thmultiskale}. We let $\varepsilon $ tend to zero in (\ref{svag
form}) and obtain%
\begin{gather}
\int_{\Omega _{T}}-u\left( x,t\right) v\left( x\right) \partial _{t}c\left(
t\right) +\left( \int_{\mathcal{Y}_{n,m}}a_{0}\left( x,t,y^{n},s^{m}\right)
dy^{n}ds^{m}\right) \cdot \nabla v\left( x\right) c\left( t\right) dxdt 
\notag \\
=\int_{\Omega _{T}}f\left( x,t\right) v\left( x\right) c\left( t\right) dxdt%
\text{,}  \label{homog_probl}
\end{gather}%
which is the homogenized problem if we can prove that 
\begin{equation*}
a_{0}\left( x,t,y^{n},s^{m}\right) =a\left( y^{n},s^{m},\nabla u\left(
x,t\right) +\sum\limits_{j=1}^{n}\nabla _{y_{j}}u_{j}\left(
x,t,y^{j},s^{m-d_{j}}\right) \right)
\end{equation*}%
with $u$\ and $u_{j}$\ as given in the theorem. To characterize $a_{0}$ we
will use the system of local problems (\ref{equation5}), and deriving this
will be our next aim.

In (\ref{svag form}) we will use test functions defined according to the
following. Let $r_{\varepsilon }=r\left( \varepsilon \right) $ be a sequence
of positive numbers tending to zero as $\varepsilon $ does. Fix $i=1,\ldots
,n$ and choose%
\begin{equation*}
v\left( x\right) =r_{\varepsilon }v_{1}\left( x\right) v_{2}\left( \frac{x}{%
\hat{\varepsilon}_{1}}\right) \cdots v_{i+1}\left( \frac{x}{\hat{\varepsilon}%
_{i}}\right)
\end{equation*}%
and%
\begin{equation*}
c\left( t\right) =c_{1}\left( t\right) c_{2}\left( \frac{t}{\check{%
\varepsilon}_{1}}\right) \cdots c_{\lambda +1}\left( \frac{t}{\check{%
\varepsilon}_{\lambda }}\right) \text{, }\lambda =1,\ldots ,m
\end{equation*}%
with $v_{1}\in D(\Omega )$,$\,v_{j}\in C_{\sharp }^{\infty }(Y_{j-1})$ for $%
j=2,\ldots ,i$, $v_{i+1}\in C_{\sharp }^{\infty }(Y_{i})/%
\mathbb{R}
$, $c_{1}\in D(0,T)$\ and$\ c_{l}\in C_{\sharp }^{\infty }(S_{l-1})$ for $%
l=2,\ldots ,\lambda +1$. We get%
\begin{gather*}
\int_{\Omega _{T}}-u^{\varepsilon }\left( x,t\right) v_{1}\left( x\right)
v_{2}\left( \frac{x}{\hat{\varepsilon}_{1}}\right) \cdots v_{i+1}\left( 
\frac{x}{\hat{\varepsilon}_{i}}\right) \\
\times \left( r_{\varepsilon }\partial _{t}c_{1}\left( t\right) c_{2}\left( 
\frac{t}{\check{\varepsilon}_{1}}\right) \cdots c_{\lambda +1}\left( \frac{t%
}{\check{\varepsilon}_{\lambda }}\right) \right. \\
+\sum_{l=2}^{\lambda +1}\left. \frac{r_{\varepsilon }}{\check{\varepsilon}%
_{l-1}}c_{1}\left( t\right) c_{2}\left( \frac{t}{\check{\varepsilon}_{1}}%
\right) \cdots \partial _{s_{l-1}}c_{l}\left( \frac{t}{\check{\varepsilon}%
_{l-1}}\right) \cdots c_{\lambda +1}\left( \frac{t}{\check{\varepsilon}%
_{\lambda }}\right) \right) \\
+a\left( \frac{x}{\hat{\varepsilon}^{n}},\frac{t}{\check{\varepsilon}^{m}}%
,\nabla u^{\varepsilon }\left( x,t\right) \right) \cdot \left(
r_{\varepsilon }\nabla v_{1}\left( x\right) v_{2}\left( \frac{x}{\hat{%
\varepsilon}_{1}}\right) \cdots v_{i+1}\left( \frac{x}{\hat{\varepsilon}_{i}}%
\right) \right. \\
+\sum_{j=2}^{i+1}\left. \frac{r_{\varepsilon }}{\hat{\varepsilon}_{j-1}}%
v_{1}\left( x\right) v_{2}\left( \frac{x}{\hat{\varepsilon}_{1}}\right)
\cdots \nabla _{y_{j-1}}v_{j}\left( \frac{x}{\hat{\varepsilon}_{j-1}}\right)
\cdots v_{i+1}\left( \frac{x}{\hat{\varepsilon}_{i}}\right) \right) \\
\times c_{1}\left( t\right) c_{2}\left( \frac{t}{\check{\varepsilon}_{1}}%
\right) \cdots c_{\lambda +1}\left( \frac{t}{\check{\varepsilon}_{\lambda }}%
\right) dxdt \\
=\int_{\Omega _{T}}f\left( x,t\right) r_{\varepsilon }v_{1}\left( x\right)
v_{2}\left( \frac{x}{\hat{\varepsilon}_{1}}\right) \cdots v_{i+1}\left( 
\frac{x}{\hat{\varepsilon}_{i}}\right) \\
\times c_{1}\left( t\right) c_{2}\left( \frac{t}{\check{\varepsilon}_{1}}%
\right) \cdots c_{\lambda +1}\left( \frac{t}{\check{\varepsilon}_{\lambda }}%
\right) dxdt.
\end{gather*}%
Applying Theorem \ref{gradkarmulti} and the definition of $r_{\varepsilon }$%
, we may let $\varepsilon \rightarrow 0$ and get 
\begin{gather*}
\lim_{\varepsilon \rightarrow 0}\int_{\Omega _{T}}-u^{\varepsilon }\left(
x,t\right) v_{1}\left( x\right) v_{2}\left( \frac{x}{\hat{\varepsilon}_{1}}%
\right) \cdots v_{i+1}\left( \frac{x}{\hat{\varepsilon}_{i}}\right) \\
\times \left( \sum_{l=2}^{\lambda +1}\frac{r_{\varepsilon }}{\check{%
\varepsilon}_{l-1}}c_{1}\left( t\right) c_{2}\left( \frac{t}{\check{%
\varepsilon}_{1}}\right) \cdots \partial _{s_{l-1}}c_{l}\left( \frac{t}{%
\check{\varepsilon}_{l-1}}\right) \cdots c_{\lambda +1}\left( \frac{t}{%
\check{\varepsilon}_{\lambda }}\right) \right) \\
+a\left( \frac{x}{\hat{\varepsilon}^{n}},\frac{t}{\check{\varepsilon}^{m}}%
,\nabla u^{\varepsilon }\left( x,t\right) \right) \\
\cdot \sum_{j=2}^{i+1}\frac{r_{\varepsilon }}{\hat{\varepsilon}_{j-1}}%
v_{1}\left( x\right) v_{2}\left( \frac{x}{\hat{\varepsilon}_{1}}\right)
\cdots \nabla _{y_{j-1}}v_{j}\left( \frac{x}{\hat{\varepsilon}_{j-1}}\right)
\cdots v_{i+1}\left( \frac{x}{\hat{\varepsilon}_{i}}\right) \\
\times c_{1}\left( t\right) c_{2}\left( \frac{t}{\check{\varepsilon}_{1}}%
\right) \cdots c_{\lambda +1}\left( \frac{t}{\check{\varepsilon}_{\lambda }}%
\right) dxdt=0
\end{gather*}%
if we omit the terms passing to zero. Rewriting we obtain%
\begin{gather}
\lim_{\varepsilon \rightarrow 0}\int_{\Omega _{T}}-\frac{1}{\hat{\varepsilon}%
_{i}}u^{\varepsilon }\left( x,t\right) \sum\limits_{l=2}^{\lambda +1}\frac{%
r_{\varepsilon }\hat{\varepsilon}_{i}}{\check{\varepsilon}_{l-1}}v_{1}\left(
x\right) v_{2}\left( \frac{x}{\hat{\varepsilon}_{1}}\right) \cdots
v_{i+1}\left( \frac{x}{\hat{\varepsilon}_{i}}\right)  \notag \\
\times c_{1}\left( t\right) c_{2}\left( \frac{t}{\check{\varepsilon}_{1}}%
\right) \cdots \partial _{s_{l-1}}c_{l}\left( \frac{t}{\check{\varepsilon}%
_{l-1}}\right) \cdots c_{\lambda +1}\left( \frac{t}{\check{\varepsilon}%
_{\lambda }}\right)  \notag \\
+a\left( \frac{x}{\hat{\varepsilon}^{n}},\frac{t}{\check{\varepsilon}^{m}}%
,\nabla u^{\varepsilon }\left( x,t\right) \right)  \label{3equat} \\
\cdot \sum\limits_{j=2}^{i+1}\frac{r_{\varepsilon }}{\hat{\varepsilon}_{j-1}}%
v_{1}\left( x\right) v_{2}\left( \frac{x}{\hat{\varepsilon}_{1}}\right)
\cdots \nabla _{y_{j-1}}v_{j}\left( \frac{x}{\hat{\varepsilon}_{j-1}}\right)
\cdots v_{i+1}\left( \frac{x}{\hat{\varepsilon}_{i}}\right)  \notag \\
\times c_{1}\left( t\right) c_{2}\left( \frac{t}{\check{\varepsilon}_{1}}%
\right) \cdots c_{\lambda +1}\left( \frac{t}{\check{\varepsilon}_{\lambda }}%
\right) ~dxdt=0,  \notag
\end{gather}%
where we have factored out $\frac{1}{\hat{\varepsilon}_{i}}$ from the first
sum to make it obvious that it is possible to pass to the limit by means of
very weak $\left( i+1,\lambda +1\right) $-scale convergence. Suppose that $%
\left\{ \frac{r_{\varepsilon }\hat{\varepsilon}_{i}}{\check{\varepsilon}%
_{\lambda }}\right\} $ and $\left\{ \frac{r_{\varepsilon }}{\hat{\varepsilon}%
_{i}}\right\} $\ are bounded. This implies that%
\begin{equation*}
\frac{r_{\varepsilon }\hat{\varepsilon}_{i}}{\check{\varepsilon}_{\lambda -j}%
}\rightarrow 0\text{, }j=1,\ldots ,\lambda -1
\end{equation*}%
and%
\begin{equation*}
\frac{r_{\varepsilon }}{\hat{\varepsilon}_{i-j}}\rightarrow 0\text{, }%
j=1,\ldots ,i-1
\end{equation*}%
as $\varepsilon \rightarrow 0$ due to the fact that the scales are
separated. Hence, under these assumptions (\ref{3equat}) turns into%
\begin{gather}
\lim_{\varepsilon \rightarrow 0}\int_{\Omega _{T}}-\frac{1}{\hat{\varepsilon}%
_{i}}u^{\varepsilon }\left( x,t\right) \frac{r_{\varepsilon }\hat{\varepsilon%
}_{i}}{\check{\varepsilon}_{\lambda }}v_{1}\left( x\right) v_{2}\left( \frac{%
x}{\hat{\varepsilon}_{1}}\right) \cdots v_{i+1}\left( \frac{x}{\hat{%
\varepsilon}_{i}}\right)  \notag \\
\times c_{1}\left( t\right) c_{2}\left( \frac{t}{\check{\varepsilon}_{1}}%
\right) \cdots \partial _{s_{\lambda }}c_{\lambda +1}\left( \frac{t}{\check{%
\varepsilon}_{\lambda }}\right) +a\left( \frac{x}{\hat{\varepsilon}^{n}},%
\frac{t}{\check{\varepsilon}^{m}},\nabla u^{\varepsilon }\left( x,t\right)
\right)  \label{springboard} \\
\cdot \frac{r_{\varepsilon }}{\hat{\varepsilon}_{i}}v_{1}\left( x\right)
v_{2}\left( \frac{x}{\hat{\varepsilon}_{1}}\right) \cdots v_{i}\left( \frac{x%
}{\hat{\varepsilon}_{i-1}}\right) \nabla _{y_{i}}v_{i+1}\left( \frac{x}{\hat{%
\varepsilon}_{i}}\right)  \notag \\
\times c_{1}\left( t\right) c_{2}\left( \frac{t}{\check{\varepsilon}_{1}}%
\right) \cdots c_{\lambda +1}\left( \frac{t}{\check{\varepsilon}_{\lambda }}%
\right) dxdt=0  \notag
\end{gather}%
which will be our springboard when deriving both the independencies of the
local time variables in the corrector functions and the local problems. This
will be done for the two different cases nonresonance and resonance.

\textbf{Case 1:} Nonresonance ($\rho _{i}=0$). First we derive the
independencies for $d_{i}>0$. Let $\lambda $ successively be $m,\ldots
,m-d_{i}+1$. If $r_{\varepsilon }=\frac{\check{\varepsilon}_{\lambda }}{\hat{%
\varepsilon}_{i}}$\ we have from the chosen values of $\lambda $ and the
meaning of $d_{i}$ that%
\begin{equation*}
\frac{r_{\varepsilon }\hat{\varepsilon}_{i}}{\check{\varepsilon}_{\lambda }}%
=1
\end{equation*}%
and%
\begin{equation}
\frac{r_{\varepsilon }}{\hat{\varepsilon}_{i}}=\frac{\check{\varepsilon}%
_{\lambda }}{\left( \hat{\varepsilon}_{i}\right) ^{2}}\rightarrow 0
\label{two differ cas2}
\end{equation}%
as $\varepsilon \rightarrow 0$. Hence, we may use (\ref{springboard}) for
this choice of $r_{\varepsilon }$ and we have 
\begin{gather*}
\lim_{\varepsilon \rightarrow 0}\int_{\Omega _{T}}-\frac{1}{\hat{\varepsilon}%
_{i}}u^{\varepsilon }\left( x,t\right) v_{1}\left( x\right) v_{2}\left( 
\frac{x}{\hat{\varepsilon}_{1}}\right) \cdots v_{i+1}\left( \frac{x}{\hat{%
\varepsilon}_{i}}\right) \\
\times c_{1}\left( t\right) c_{2}\left( \frac{t}{\check{\varepsilon}_{1}}%
\right) \cdots \partial _{s_{\lambda }}c_{\lambda +1}\left( \frac{t}{\check{%
\varepsilon}_{\lambda }}\right) +a\left( \frac{x}{\hat{\varepsilon}^{n}},%
\frac{t}{\check{\varepsilon}^{m}},\nabla u^{\varepsilon }\left( x,t\right)
\right) \\
\cdot \frac{\check{\varepsilon}_{\lambda }}{\left( \hat{\varepsilon}%
_{i}\right) ^{2}}v_{1}\left( x\right) v_{2}\left( \frac{x}{\hat{\varepsilon}%
_{1}}\right) \cdots v_{i}\left( \frac{x}{\hat{\varepsilon}_{i-1}}\right)
\nabla _{y_{i}}v_{i+1}\left( \frac{x}{\hat{\varepsilon}_{i}}\right) \\
\times c_{1}\left( t\right) c_{2}\left( \frac{t}{\check{\varepsilon}_{1}}%
\right) \cdots c_{\lambda +1}\left( \frac{t}{\check{\varepsilon}_{\lambda }}%
\right) dxdt=0\text{.}
\end{gather*}%
We let $\varepsilon $ tend to zero and obtain, due to Theorem \ref{T vw} and
(\ref{two differ cas2}), that%
\begin{gather*}
\int_{\Omega _{T}}\int_{\mathcal{Y}_{i,\lambda }}-u_{i}\left(
x,t,y^{i},s^{\lambda }\right) v_{1}\left( x\right) v_{2}\left( y_{1}\right)
\cdots v_{i+1}\left( y_{i}\right) \\
\times c_{1}\left( t\right) c_{2}\left( s_{1}\right) \cdots \partial
_{s_{\lambda }}c_{\lambda +1}\left( s_{\lambda }\right) dy^{i}ds^{\lambda
}dxdt=0
\end{gather*}%
and by the variational lemma we have%
\begin{equation*}
\int_{S_{\lambda }}-u_{i}\left( x,t,y^{i},s^{\lambda }\right) \partial
_{s_{\lambda }}c_{\lambda +1}\left( s_{\lambda }\right) ds_{\lambda }=0
\end{equation*}%
almost everywhere for all $c_{\lambda +1}\in C_{\sharp }^{\infty
}(S_{\lambda })$. This means that $u_{i}$ is independent of $%
s_{m-d_{i}+1},\ldots ,s_{m}$.

We proceed by deriving the local problems and for this purpose we choose $%
r_{\varepsilon }=\hat{\varepsilon}_{i}$ and $\lambda =m-d_{i},$ where $%
d_{i}\geq 0$. Since $d_{i}\geq 0$ and $\rho _{i}=0$ we conclude that 
\begin{equation*}
\frac{r_{\varepsilon }\hat{\varepsilon}_{i}}{\check{\varepsilon}_{\lambda }}=%
\frac{\left( \hat{\varepsilon}_{i}\right) ^{2}}{\check{\varepsilon}_{m-d_{i}}%
}\rightarrow 0
\end{equation*}%
as $\varepsilon \rightarrow 0$ and 
\begin{equation*}
\frac{r_{\varepsilon }}{\hat{\varepsilon}_{i}}=1\text{,}
\end{equation*}%
which means that (\ref{springboard}) is valid and we get 
\begin{gather*}
\lim_{\varepsilon \rightarrow 0}\int_{\Omega _{T}}-\frac{1}{\hat{\varepsilon}%
_{i}}u^{\varepsilon }\left( x,t\right) \frac{\left( \hat{\varepsilon}%
_{i}\right) ^{2}}{\check{\varepsilon}_{m-d_{i}}}v_{1}\left( x\right)
v_{2}\left( \frac{x}{\hat{\varepsilon}_{1}}\right) \cdots v_{i+1}\left( 
\frac{x}{\hat{\varepsilon}_{i}}\right) \\
\times c_{1}\left( t\right) c_{2}\left( \frac{t}{\check{\varepsilon}_{1}}%
\right) \cdots \partial _{s_{m-d_{i}}}c_{m-d_{i}+1}\left( \frac{t}{\check{%
\varepsilon}_{m-d_{i}}}\right) +a\left( \frac{x}{\hat{\varepsilon}^{n}},%
\frac{t}{\check{\varepsilon}^{m}},\nabla u^{\varepsilon }\left( x,t\right)
\right) \\
\cdot v_{1}\left( x\right) v_{2}\left( \frac{x}{\hat{\varepsilon}_{1}}%
\right) \cdots v_{i}\left( \frac{x}{\hat{\varepsilon}_{i-1}}\right) \nabla
_{y_{i}}v_{i+1}\left( \frac{x}{\hat{\varepsilon}_{i}}\right) \\
\times c_{1}\left( t\right) c_{2}\left( \frac{t}{\check{\varepsilon}_{1}}%
\right) \cdots c_{m-d_{i}+1}\left( \frac{t}{\check{\varepsilon}_{m-d_{i}}}%
\right) dxdt=0\text{.}
\end{gather*}%
As $\varepsilon \rightarrow 0$ we obtain%
\begin{gather*}
\int_{\Omega _{T}}\int_{\mathcal{Y}_{n,m}}a_{0}\left( x,t,y^{n},s^{m}\right)
\\
\cdot v_{1}\left( x\right) v_{2}\left( y_{1}\right) \cdots v_{i}\left(
y_{i-1}\right) \nabla _{y_{i}}v_{i+1}\left( y_{i}\right) \\
\times c_{1}\left( t\right) c_{2}\left( s_{1}\right) \cdots
c_{m-d_{i}+1}\left( s_{m-d_{i}}\right) dy^{n}ds^{m}dxdt=0
\end{gather*}%
and, finally,%
\begin{gather}
\int_{S_{m-d_{i}+1}}\cdots \int_{S_{m}}\int_{Y_{i}}\cdots
\int_{Y_{n}}a_{0}\left( x,t,y^{n},s^{m}\right)  \label{case 1} \\
\cdot \nabla _{y_{i}}v_{i+1}\left( y_{i}\right) dy_{n}\cdots
dy_{i}ds_{m}\cdots ds_{m-d_{i}+1}=0  \notag
\end{gather}%
almost everywhere for all $v_{i+1}\in H_{\sharp }^{1}(Y_{i})/%
\mathbb{R}
$, which is the weak form of the local problem in this nonresonance case.

\textbf{Case 2:} Resonance ($\rho _{i}=C$). As in the first case we begin
with the independencies for $d_{i}>0$. Again, let $\lambda $ successively be 
$m,\ldots ,m-d_{i}+1$. Now choose $r_{\varepsilon }=\frac{\check{\varepsilon}%
_{\lambda }}{\hat{\varepsilon}_{i}}$ directly implying that%
\begin{equation*}
\frac{r_{\varepsilon }\hat{\varepsilon}_{i}}{\check{\varepsilon}_{\lambda }}%
=1
\end{equation*}%
and 
\begin{equation*}
\frac{r_{\varepsilon }}{\hat{\varepsilon}_{i}}=\frac{\check{\varepsilon}%
_{\lambda }}{\left( \hat{\varepsilon}_{i}\right) ^{2}}\rightarrow 0
\end{equation*}%
when $\varepsilon \rightarrow 0$, by the restriction of $\lambda $ and the
definition of $d_{i}$ and $\rho _{i}$. Thus, (\ref{springboard}) turns into%
\begin{gather*}
\lim_{\varepsilon \rightarrow 0}\int_{\Omega _{T}}-\frac{1}{\hat{\varepsilon}%
_{i}}u^{\varepsilon }\left( x,t\right) v_{1}\left( x\right) v_{2}\left( 
\frac{x}{\hat{\varepsilon}_{1}}\right) \cdots v_{i+1}\left( \frac{x}{\hat{%
\varepsilon}_{i}}\right) \\
\times c_{1}\left( t\right) c_{2}\left( \frac{t}{\check{\varepsilon}_{1}}%
\right) \cdots \partial _{s_{\lambda }}c_{\lambda +1}\left( \frac{t}{\check{%
\varepsilon}_{\lambda }}\right) +a\left( \frac{x}{\hat{\varepsilon}^{n}},%
\frac{t}{\check{\varepsilon}^{m}},\nabla u^{\varepsilon }\left( x,t\right)
\right) \\
\cdot \frac{\check{\varepsilon}_{\lambda }}{\left( \hat{\varepsilon}%
_{i}\right) ^{2}}v_{1}\left( x\right) v_{2}\left( \frac{x}{\hat{\varepsilon}%
_{1}}\right) \cdots v_{i}\left( \frac{x}{\hat{\varepsilon}_{i-1}}\right)
\nabla _{y_{i}}v_{i+1}\left( \frac{x}{\hat{\varepsilon}_{i}}\right) \\
\times c_{1}\left( t\right) c_{2}\left( \frac{t}{\check{\varepsilon}_{1}}%
\right) \cdots c_{\lambda +1}\left( \frac{t}{\check{\varepsilon}_{\lambda }}%
\right) dxdt=0
\end{gather*}%
and a passage to the limit gives%
\begin{gather*}
\int_{\Omega _{T}}\int_{\mathcal{Y}_{i,\lambda }}-u_{i}\left(
x,t,y^{i},s^{\lambda }\right) v_{1}\left( x\right) v_{2}\left( y_{1}\right)
\cdots v_{i+1}\left( y_{i}\right) \\
\times c_{1}\left( t\right) c_{2}\left( s_{1}\right) \cdots \partial
_{s_{\lambda }}c_{\lambda +1}\left( s_{\lambda }\right) dy^{i}ds^{\lambda
}dxdt=0\text{.}
\end{gather*}%
Hence,%
\begin{equation*}
\int_{S_{\lambda }}-u_{i}\left( x,t,y^{i},s^{\lambda }\right) \partial
_{s_{\lambda }}c_{\lambda +1}\left( s_{\lambda }\right) ds_{\lambda }=0
\end{equation*}%
almost everywhere for all $c_{\lambda +1}\in C_{\sharp }^{\infty
}(S_{\lambda })$, and thus $u_{i}$ is independent of $s_{\lambda }$.

To extract the local problem we choose $r_{\varepsilon }=\hat{\varepsilon}%
_{i}$ and $\lambda =m-d_{i}$, where $d_{i}\geq 0$, which gives 
\begin{equation*}
\frac{r_{\varepsilon }\hat{\varepsilon}_{i}}{\check{\varepsilon}_{\lambda }}=%
\frac{\left( \hat{\varepsilon}_{i}\right) ^{2}}{\check{\varepsilon}_{m-d_{i}}%
}\rightarrow \rho _{i}
\end{equation*}%
as $\varepsilon \rightarrow 0$ and 
\begin{equation*}
\frac{r_{\varepsilon }}{\hat{\varepsilon}_{i}}=1
\end{equation*}%
and from (\ref{springboard}) we then have%
\begin{gather*}
\lim_{\varepsilon \rightarrow 0}\int_{\Omega _{T}}-\frac{1}{\hat{\varepsilon}%
_{i}}u^{\varepsilon }\left( x,t\right) \frac{\left( \hat{\varepsilon}%
_{i}\right) ^{2}}{\check{\varepsilon}_{m-d_{i}}}v_{1}\left( x\right)
v_{2}\left( \frac{x}{\hat{\varepsilon}_{1}}\right) \cdots v_{i+1}\left( 
\frac{x}{\hat{\varepsilon}_{i}}\right) \\
\times c_{1}\left( t\right) c_{2}\left( \frac{t}{\check{\varepsilon}_{1}}%
\right) \cdots \partial _{s_{m-d_{i}}}c_{m-d_{i}+1}\left( \frac{t}{\check{%
\varepsilon}_{m-d_{i}}}\right) +a\left( \frac{x}{\hat{\varepsilon}^{n}},%
\frac{t}{\check{\varepsilon}^{m}},\nabla u^{\varepsilon }\left( x,t\right)
\right) \\
\cdot v_{1}\left( x\right) v_{2}\left( \frac{x}{\hat{\varepsilon}_{1}}%
\right) \cdots v_{i}\left( \frac{x}{\hat{\varepsilon}_{i-1}}\right) \nabla
_{y_{i}}v_{i+1}\left( \frac{x}{\hat{\varepsilon}_{i}}\right) \\
\times c_{1}\left( t\right) c_{2}\left( \frac{t}{\check{\varepsilon}_{1}}%
\right) \cdots c_{m-d_{i}+1}\left( \frac{t}{\check{\varepsilon}_{m-d_{i}}}%
\right) dxdt=0\text{.}
\end{gather*}%
Letting $\varepsilon $ tend to zero and applying Theorem \ref{T vw} we
obtain 
\begin{gather*}
\int_{\Omega _{T}}\int_{\mathcal{Y}_{n,m}}-\rho _{i}u_{i}\left(
x,t,y^{i},s^{m-d_{i}}\right) v_{1}\left( x\right) v_{2}\left( y_{1}\right)
\cdots v_{i+1}\left( y_{i}\right) \\
\times c_{1}\left( t\right) c_{2}\left( s_{1}\right) \cdots \partial
_{s_{m-d_{i}}}c_{m-d_{i}+1}\left( s_{m-d_{i}}\right) +a_{0}\left(
x,t,y^{n},s^{m}\right) \\
\cdot v_{1}\left( x\right) v_{2}\left( y_{1}\right) \cdots v_{i}\left(
y_{i-1}\right) \nabla _{y_{i}}v_{i+1}\left( y_{i}\right) \\
\times c_{1}\left( t\right) c_{2}\left( s_{1}\right) \cdots
c_{m-d_{i}+1}\left( s_{m-d_{i}}\right) dy^{n}ds^{m}dxdt=0
\end{gather*}%
and hence, we end up with%
\begin{gather}
\int_{S_{m-d_{i}}}\cdots \int_{S_{m}}\int_{Y_{i}}\cdots \int_{Y_{n}}-\rho
_{i}u_{i}\left( x,t,y^{i},s^{m-d_{i}}\right) v_{i+1}\left( y_{i}\right) 
\notag \\
\times \partial _{s_{m-d_{i}}}c_{m-d_{i}+1}\left( s_{m-d_{i}}\right)
+a_{0}\left( x,t,y^{n},s^{m}\right)  \label{locCas2} \\
\cdot \nabla _{y_{i}}v_{i+1}\left( y_{i}\right) c_{m-d_{i}+1}\left(
s_{m-d_{i}}\right) dy_{n}\cdots dy_{i}ds_{m}\cdots ds_{m-d_{i}}=0  \notag
\end{gather}%
almost everywhere for all $v_{i+1}\in H_{\sharp }^{1}(Y_{i})/%
\mathbb{R}
$ and $c_{m-d_{i}+1}\in C_{\sharp }^{\infty }(S_{m-d_{i}})$, the weak form
of the local problem in this second case.

What remains is to characterize $a_{0}$ and to this end we use perturbed
test functions, see \cite{Ev1} and \cite{Ev2}, according to%
\begin{gather*}
p^{k}\left( x,t,y^{j},s^{m}\right) \\
=p^{k,0}\left( x,t\right) +\sum_{j=1}^{n}p^{k,j}\left(
x,t,y^{j},s^{m-d_{j}}\right) +\delta c\left( x,t,y^{n},s^{m}\right) \text{,}
\end{gather*}%
where $p^{k,0}\in D(\Omega _{T})^{N}$, $p^{k,j}\in D(\Omega _{T};C_{\sharp
}^{\infty }(\mathcal{Y}_{j,m-d_{j}}))^{N}$ for $j=1,\ldots ,n$, $c\in
D(\Omega _{T};C_{\sharp }^{\infty }(\mathcal{Y}_{n,m}))^{N}$ and $\delta >0$%
.\ We choose these sequences such that%
\begin{equation*}
p^{k,0}\left( x,t\right) \rightarrow \nabla u\left( x,t\right) \text{ in }%
L^{2}(\Omega _{T})^{N}\text{,}
\end{equation*}%
\begin{equation*}
p^{k,j}\left( x,t,y^{j},s^{m-d_{j}}\right) \rightarrow \nabla
_{y_{j}}u_{j}\left( x,t,y^{j},s^{m-d_{j}}\right) \text{ in }L^{2}(\Omega
_{T}\times \mathcal{Y}_{j,m-d_{j}})^{N}
\end{equation*}%
and such that they converge almost everywhere to the same limits as $%
k\rightarrow \infty $, see p. 388 in \cite{Kuf}. We introduce the notation 
\begin{equation*}
p_{\varepsilon }^{k}\left( x,t\right) =p^{k}\left( x,t,\frac{x}{\hat{%
\varepsilon}^{n}},\frac{t}{\check{\varepsilon}^{m}}\right) \text{.}
\end{equation*}%
Using property (iv)\ we get 
\begin{equation*}
\left( a\left( \frac{x}{\hat{\varepsilon}^{n}},\frac{t}{\check{\varepsilon}%
^{m}},\nabla u^{\varepsilon }\right) -a\left( \frac{x}{\hat{\varepsilon}^{n}}%
,\frac{t}{\check{\varepsilon}^{m}},p_{\varepsilon }^{k}\right) \right) \cdot
(\nabla u^{\varepsilon }\left( x,t\right) -p_{\varepsilon }^{k}\left(
x,t\right) )\geq 0
\end{equation*}%
and integration and expansion leads to%
\begin{gather}
\int_{\Omega _{T}}a\left( \frac{x}{\hat{\varepsilon}^{n}},\frac{t}{\check{%
\varepsilon}^{m}},\nabla u^{\varepsilon }\right) \cdot \nabla u^{\varepsilon
}\left( x,t\right) -a\left( \frac{x}{\hat{\varepsilon}^{n}},\frac{t}{\check{%
\varepsilon}^{m}},\nabla u^{\varepsilon }\right) \cdot p_{\varepsilon
}^{k}\left( x,t\right)  \label{star2} \\
-a\left( \frac{x}{\hat{\varepsilon}^{n}},\frac{t}{\check{\varepsilon}^{m}}%
,p_{\varepsilon }^{k}\right) \cdot \nabla u^{\varepsilon }\left( x,t\right)
+a\left( \frac{x}{\hat{\varepsilon}^{n}},\frac{t}{\check{\varepsilon}^{m}}%
,p_{\varepsilon }^{k}\right) \cdot p_{\varepsilon }^{k}\left( x,t\right)
dxdt\geq 0\text{.}  \notag
\end{gather}%
Due to Theorem 30.A (c)\ in \cite{ZeiIIB} we may replace $vc$ with $%
u^{\varepsilon }$\ in\ (\ref{svag form}) and get another way of expressing
the first term in (\ref{star2})\emph{\ }and hence\ it can be written as%
\begin{gather}
\int_{\Omega _{T}}f\left( x,t\right) u^{\varepsilon }\left( x,t\right)
-a\left( \frac{x}{\hat{\varepsilon}^{n}},\frac{t}{\check{\varepsilon}^{m}}%
,\nabla u^{\varepsilon }\right) \cdot p_{\varepsilon }^{k}\left( x,t\right) 
\notag \\
-a\left( \frac{x}{\hat{\varepsilon}^{n}},\frac{t}{\check{\varepsilon}^{m}}%
,p_{\varepsilon }^{k}\right) \cdot \nabla u^{\varepsilon }\left( x,t\right)
+a\left( \frac{x}{\hat{\varepsilon}^{n}},\frac{t}{\check{\varepsilon}^{m}}%
,p_{\varepsilon }^{k}\right) \cdot p_{\varepsilon }^{k}\left( x,t\right) dxdt
\label{bluttan} \\
-\int_{0}^{T}\left\langle \partial _{t}u^{\varepsilon }\left( t\right)
,u^{\varepsilon }\left( t\right) \right\rangle _{H^{-1}(\Omega
),H_{0}^{1}(\Omega )}dt\geq 0\text{.}  \notag
\end{gather}%
We note that $p^{k}$, $a\left( y^{n},s^{m},p^{k}\right) $ and their product
are admissible test functions and since 
\begin{equation*}
-\lim_{\varepsilon \rightarrow 0}\inf \!\int_{0}^{T}\!\!\left\langle
\partial _{t}u^{\varepsilon }\left( t\right) ,u^{\varepsilon }\left(
t\right) \right\rangle _{H^{-1}(\Omega ),H_{0}^{1}(\Omega )}dt\leq
-\!\int_{0}^{T}\!\!\left\langle \partial _{t}u\left( t\right) ,u\left(
t\right) \right\rangle _{H^{-1}(\Omega ),H_{0}^{1}(\Omega )}\!dt
\end{equation*}%
(see p. 12--13 in \cite{NgWo1}) we get, up to a subsequence, that%
\begin{gather}
\int_{\Omega _{T}}\int_{\mathcal{Y}_{n,m}}f\left( x,t\right) u\left(
x,t\right) -a_{0}\left( x,t,y^{n},s^{m}\right) \cdot p^{k}\left(
x,t,y^{n},s^{m}\right)  \notag \\
-a\left( y^{n},s^{m},p^{k}\right) \cdot \left( \nabla u\left( x,t\right)
+\sum\limits_{j=1}^{n}\nabla _{y_{j}}u_{j}\left(
x,t,y^{j},s^{m-d_{j}}\right) \right)  \notag \\
+a\left( y^{n},s^{m},p^{k}\right) \cdot p^{k}\left( x,t,y^{n},s^{m}\right)
dy^{n}ds^{m}dxdt  \label{epsigrans} \\
-\int_{0}^{T}\left\langle \partial _{t}u\left( t\right) ,u\left( t\right)
\right\rangle _{H^{-1}(\Omega ),H_{0}^{1}(\Omega )}dt\geq 0  \notag
\end{gather}%
when $\varepsilon $ tends to zero. We proceed by letting\emph{\ }$k$ tend to
infinity. From the choice of $p^{k}$ we have that%
\begin{equation*}
p^{k}\left( x,t,y^{n},s^{m}\right) \rightarrow \nabla u\left( x,t\right)
+\sum\limits_{j=1}^{n}\nabla _{y_{j}}u_{j}\left(
x,t,y^{j},s^{m-d_{j}}\right) +\delta c\left( x,t,y^{n},s^{m}\right)
\end{equation*}%
in $L^{2}(\Omega _{T}\times \mathcal{Y}_{n,m})^{N}$ and almost everywhere in 
$\Omega _{T}\times \mathcal{Y}_{n,m}$. Furthermore%
\begin{equation*}
a\left( y^{n},s^{m},p^{k}\right) \rightarrow a\left( y^{n},s^{m},\nabla
u+\sum\limits_{j=1}^{n}\nabla _{y_{j}}u_{j}+\delta c\right)
\end{equation*}%
almost everywhere in $\Omega _{T}\times \mathcal{Y}_{n,m}$ and hence%
\begin{gather*}
a\left( y^{n},s^{m},p^{k}\right) \cdot p^{k}\left( x,t,y^{n},s^{m}\right)
\rightarrow a\left( y^{n},s^{m},\nabla u+\sum\limits_{j=1}^{n}\nabla
_{y_{j}}u_{j}+\delta c\right) \\
\cdot \left( \nabla u\left( x,t\right) +\sum\limits_{j=1}^{n}\nabla
_{y_{j}}u_{j}\left( x,t,y^{j},s^{m-d_{j}}\right) +\delta c\left(
x,t,y^{n},s^{m}\right) \right)
\end{gather*}%
almost everywhere in $\Omega _{T}\times \mathcal{Y}_{n,m}$. When we pass to
the limit in (\ref{epsigrans}) we will use Lebesgue's generalized majorized
convergence theorem (Theorem (19a) p. 1015 in \cite{ZeiIIB}) for the third
and fourth term where we go through the details for the fourth term.
Choosing $\xi =p^{k}$ in (\ref{flyttning}) we have that%
\begin{equation}
\left\vert a\left( y^{n},s^{m},p^{k}\right) \right\vert \leq C_{1}\left(
1+\left\vert p^{k}\left( x,t,y^{n},s^{m}\right) \right\vert \right) \text{.}
\label{begr}
\end{equation}%
Successively applying Cauchy-Schwarz inequality and (\ref{begr}) we get%
\begin{eqnarray*}
\left\vert a\left( y^{n},s^{m},p^{k}\right) \cdot p^{k}\left(
x,t,y^{n},s^{m}\right) \right\vert &\leq &\left\vert a\left(
y^{n},s^{m},p^{k}\right) \right\vert \left\vert p^{k}\left(
x,t,y^{n},s^{m}\right) \right\vert \\
&\leq &C_{1}\left( 1+\left\vert p^{k}\left( x,t,y^{n},s^{m}\right)
\right\vert \right) \left\vert p^{k}\left( x,t,y^{n},s^{m}\right) \right\vert
\\
&=&C_{1}\left( \left\vert p^{k}\left( x,t,y^{n},s^{m}\right) \right\vert
+\left\vert p^{k}\left( x,t,y^{n},s^{m}\right) \right\vert ^{2}\right) \text{%
.}
\end{eqnarray*}%
Letting $k\rightarrow \infty $, we have%
\begin{gather*}
\int_{\Omega _{T}}\int_{\mathcal{Y}_{n,m}}\left\vert p^{k}\left(
x,t,y^{n},s^{m}\right) \right\vert +\left\vert p^{k}\left(
x,t,y^{n},s^{m}\right) \right\vert ^{2}dy^{n}ds^{m}dxdt \\
\rightarrow \int_{\Omega _{T}}\int_{\mathcal{Y}_{n,m}}\left\vert \nabla
u\left( x,t\right) +\sum\limits_{j=1}^{n}\nabla _{y_{j}}u_{j}\left(
x,t,y^{j},s^{m-d_{j}}\right) +\delta c\left( x,t,y^{n},s^{m}\right)
\right\vert \\
+\left\vert \nabla u\left( x,t\right) +\sum\limits_{j=1}^{n}\nabla
_{y_{j}}u_{j}\left( x,t,y^{j},s^{m-d_{j}}\right) +\delta c\left(
x,t,y^{n},s^{m}\right) \right\vert ^{2}dy^{n}ds^{m}dxdt
\end{gather*}%
and hence, by Lebesgue's generalized majorized convergence theorem we
conclude that%
\begin{gather*}
\int_{\Omega _{T}}\int_{\mathcal{Y}_{n,m}}a\left( y^{n},s^{m},p^{k}\right)
\cdot p^{k}\left( x,t,y^{n},s^{m}\right) dy^{n}ds^{m}dxdt \\
\rightarrow \int_{\Omega _{T}}\int_{\mathcal{Y}_{n,m}}a\left(
y^{n},s^{m},\nabla u+\sum\limits_{j=1}^{n}\nabla _{y_{j}}u_{j}+\delta
c\right) \\
\cdot \left( \nabla u\left( x,t\right) +\sum\limits_{j=1}^{n}\nabla
_{y_{j}}u_{j}\left( x,t,y^{j},s^{m-d_{j}}\right) +\delta c\left(
x,t,y^{n},s^{m}\right) \right) dy^{n}ds^{m}dxdt\text{.}
\end{gather*}%
Thus, as $k$ tends to infinity in (\ref{epsigrans}) we find that%
\begin{gather*}
\int_{\Omega _{T}}\int_{\mathcal{Y}_{n,m}}f\left( x,t\right) u\left(
x,t\right) -a_{0}\left( x,t,y^{n},s^{m}\right) \\
\cdot \left( \nabla u\left( x,t\right) +\sum\limits_{j=1}^{n}\nabla
_{y_{j}}u_{j}\left( x,t,y^{j},s^{m-d_{j}}\right) +\delta c\left(
x,t,y^{n},s^{m}\right) \right) \\
-a\left( y^{n},s^{m},\nabla u+\sum\limits_{j=1}^{n}\nabla
_{y_{j}}u_{j}+\delta c\right) \\
\cdot \left( \nabla u\left( x,t\right) +\sum\limits_{j=1}^{n}\nabla
_{y_{j}}u_{j}\left( x,t,y^{j},s^{m-d_{j}}\right) \right) +a\left(
y^{n},s^{m},\nabla u+\sum\limits_{j=1}^{n}\nabla _{y_{j}}u_{j}+\delta
c\right) \\
\cdot \left( \nabla u\left( x,t\right) +\sum\limits_{j=1}^{n}\nabla
_{y_{j}}u_{j}\left( x,t,y^{j},s^{m-d_{j}}\right) +\delta c\left(
x,t,y^{n},s^{m}\right) \right) dy^{n}ds^{m}dxdt \\
-\int_{0}^{T}\left\langle \partial _{t}u\left( t\right) ,u\left( t\right)
\right\rangle _{H^{-1}(\Omega ),H_{0}^{1}(\Omega )}dt\geq 0\text{,}
\end{gather*}%
where some terms vanish directly and we have%
\begin{gather}
\int_{\Omega _{T}}\int_{\mathcal{Y}_{n,m}}f\left( x,t\right) u\left(
x,t\right) -a_{0}\left( x,t,y^{n},s^{m}\right)  \notag \\
\cdot \left( \nabla u\left( x,t\right) +\sum\limits_{j=1}^{n}\nabla
_{y_{j}}u_{j}\left( x,t,y^{j},s^{m-d_{j}}\right) +\delta c\left(
x,t,y^{n},s^{m}\right) \right)  \label{Assttar2} \\
+a\left( y^{n},s^{m},\nabla u+\sum\limits_{j=1}^{n}\nabla
_{y_{j}}u_{j}+\delta c\right) \cdot \delta c\left( x,t,y^{n},s^{m}\right)
dy^{n}ds^{m}dxdt  \notag \\
-\int_{0}^{T}\left\langle \partial _{t}u\left( t\right) ,u\left( t\right)
\right\rangle _{H^{-1}(\Omega ),H_{0}^{1}(\Omega )}dt\geq 0\text{.}  \notag
\end{gather}%
If we replace $vc$ by $u$ in (\ref{homog_probl}) we get%
\begin{gather}
\int_{0}^{T}\left\langle \partial _{t}u\left( t\right) ,u\left( t\right)
\right\rangle _{H^{-1}(\Omega ),H_{0}^{1}(\Omega )}dt  \notag \\
+\int_{\Omega _{T}}\left( \int_{\mathcal{Y}_{n,m}}a_{0}\left(
x,t,y^{n},s^{m}\right) dy^{n}ds^{m}\right) \cdot \nabla u\left( x,t\right)
dxdt  \label{Astar3} \\
=\int_{\Omega _{T}}f\left( x,t\right) u\left( x,t\right) dxdt  \notag
\end{gather}%
and with (\ref{Astar3}) in (\ref{Assttar2}) we obtain%
\begin{gather}
\int_{\Omega _{T}}\int_{\mathcal{Y}_{n,m}}\sum\limits_{j=1}^{n}-a_{0}\left(
x,t,y^{n},s^{m}\right) \cdot \nabla _{y_{j}}u_{j}\left(
x,t,y^{j},s^{m-d_{j}}\right)  \notag \\
-a_{0}\left( x,t,y^{n},s^{m}\right) \cdot \delta c\left(
x,t,y^{n},s^{m}\right)  \label{twostars} \\
+a\left( y^{n},s^{m},\nabla u+\sum\limits_{j=1}^{n}\nabla
_{y_{j}}u_{j}+\delta c\right) \cdot \delta c\left( x,t,y^{n},s^{m}\right)
dy^{n}ds^{m}dxdt\geq 0\text{.}  \notag
\end{gather}%
Using the local problems we will eliminate the first $n$ terms in (\ref%
{twostars}). We study them one at the time by letting $j$ successively be
equal to $1,\ldots ,n$. If $\rho _{j}=0$ we use the local problem (\ref{case
1}) from case 1 with $i=j$ and the corresponding term vanishes directly. If $%
\rho _{j}\neq 0$ then, by assumption, $u_{j}\in L^{2}(\Omega _{T}\times 
\mathcal{Y}_{j-1,m-d_{j}-1},W_{2_{{\large \sharp }}}^{1}(S_{m-d_{j}};H_{%
\sharp }^{1}(Y_{j})/%
\mathbb{R}
,L_{\sharp }^{2}(Y_{j})/%
\mathbb{R}
))$, which implies that $u_{j}(x,t,y^{j-1})\in W_{2_{{\large \sharp }%
}}^{1}(S_{m-d_{j}};H_{\sharp }^{1}(Y_{j})/%
\mathbb{R}
,L_{\sharp }^{2}(Y_{j})/%
\mathbb{R}
)$. Then, from (\ref{locCas2}) with $i=j$, we obtain that 
\begin{gather*}
\rho _{j}\partial _{s_{m-d_{j}}}u_{j}\left( x,t,y^{j},s^{m-d_{j}}\right) \\
=\left( \!\int_{S_{m-d_{j}+1}}\!\!\ldots
\!\int_{S_{m}}\!\int_{Y_{j+1}}\!\!\ldots \!\int_{Y_{n}}\!\!\!a_{0}\left(
x,t,y^{n},s^{m}\right) dy_{n}\cdots dy_{j+1}ds_{m}\cdots
ds_{m-d_{j}+1}\!\right) \!\nabla _{y_{j}}
\end{gather*}%
i.e. the first term in (\ref{twostars}) can be replaced with the derivative $%
\rho _{j}\partial _{s_{m-d_{j}}}u_{j}$. Thus, Corollary 4.1 in \cite{NgWo}
yields that the term in question vanishes. What remains of (\ref{twostars})
is%
\begin{gather*}
\int_{\Omega _{T}}\int_{\mathcal{Y}_{n,m}}\left( -a_{0}\left(
x,t,y^{n},s^{m}\right) +a\left( y^{n},s^{m},\nabla
u+\sum\limits_{j=1}^{n}\nabla _{y_{j}}u_{j}+\delta c\right) \right) \\
\cdot \delta c\left( x,t,y^{n},s^{m}\right) dy^{n}ds^{m}dxdt\geq 0\text{.}
\end{gather*}%
Dividing by $\delta $ and passing to the limit in the sense of letting $%
\delta $ tend to zero, we deduce that%
\begin{equation*}
a_{0}\left( x,t,y^{n},s^{m}\right) =a\left( y^{n},s^{m},\nabla
u+\sum\limits_{j=1}^{n}\nabla _{y_{j}}u_{j}\right) \text{.}
\end{equation*}%
Finally, by the uniqueness of $u$, the whole sequence converges and the
proof is complete.
\end{proof}

\section{An illustrative example}

In this section we investigate a specific nonlinear parabolic problem with a
number of rapid spatial and temporal scales, some of which are not powers of 
$\varepsilon $. More precisely we consider the (3,4)-scaled problem%
\begin{equation*}
\begin{array}[t]{rcl}
\!\!\partial _{t}u^{\varepsilon }\left( x,t\right) \!-\!\nabla \!\cdot
a\!\left( \frac{x}{2\sqrt{\varepsilon }},\!\frac{x}{\varepsilon ^{2}},\!%
\frac{t}{e^{\varepsilon }-1},\!\frac{t}{\ln (1+\varepsilon ^{2})},\!\frac{t}{%
\varepsilon ^{3}\ln \left( 1+\frac{1}{\varepsilon }\right) },\!\nabla
u^{\varepsilon }\left( x,t\right) \right) \!\!\!\! & \!=\!\!\! & 
\!\!\!f\left( x,t\right) \text{ in\ }\Omega _{T}\text{,} \\ 
u^{\varepsilon }\left( x,t\right) \!\!\!\! & \,\!{}\!=\! & \!\!\!0\text{ on\ 
}\partial \Omega \!\!\times \!\!(0,T)\text{,}\! \\ 
u^{\varepsilon }\left( x,0\right) \!\!\!\! & \,\!=\! & \!\!\!u^{0}\left(
x\right) \text{ in }\Omega \text{.}%
\end{array}%
\end{equation*}%
To apply Theorem \ref{first inline} we must be reassured that\ the two lists 
$\left\{ 2\sqrt{\varepsilon },\varepsilon ^{2}\right\} $ and $\left\{
e^{\varepsilon }-1,\ln (1+\varepsilon ^{2}),\varepsilon ^{3}\ln \left( 1+%
\frac{1}{\varepsilon }\right) \right\} $ are jointly well-separated. It
holds that\ 
\begin{equation*}
\lim_{\varepsilon \rightarrow 0}\frac{1}{2\sqrt{\varepsilon }}\left( \frac{%
\varepsilon ^{2}}{2\sqrt{\varepsilon }}\right) ^{1}=0\text{,}
\end{equation*}%
\begin{equation*}
\lim_{\varepsilon \rightarrow 0}\frac{1}{e^{\varepsilon }-1}\left( \frac{\ln
(1+\varepsilon ^{2})}{e^{\varepsilon }-1}\right) ^{3}=0
\end{equation*}%
and%
\begin{equation*}
\lim_{\varepsilon \rightarrow 0}\frac{1}{\ln (1+\varepsilon ^{2})}\left( 
\frac{\varepsilon ^{3}\ln \left( 1+\frac{1}{\varepsilon }\right) }{\ln
(1+\varepsilon ^{2})}\right) ^{3}=0\text{,}
\end{equation*}%
which implies that both the spatial and temporal scales are well-separated.
Moreover,%
\begin{equation*}
\lim_{\varepsilon \rightarrow 0}\frac{\ln (1+\varepsilon ^{2})}{\varepsilon
^{2}}=1\text{,}
\end{equation*}%
so we can remove duplicates and make the joint list $\!\left\{ 2\sqrt{%
\varepsilon },e^{\varepsilon }\!-\!1,\varepsilon ^{2},\varepsilon ^{3}\ln
\left( 1+\frac{1}{\varepsilon }\right) \right\} \!$, which is
well-separated. According to Definition \ref{Lists of scales} this shows
that our lists of scales are jointly well-separated. For the rest we assume
that our problem fulfils the assumptions of Theorem \ref{first inline}.

To begin with, from Theorem \ref{first inline} we know that the convergence
results (\ref{equation 2}) and (\ref{equation3}) hold, i.e. that%
\begin{equation*}
u^{\varepsilon }\left( x,t\right) \rightarrow u\left( x,t\right) \text{ in }%
L^{2}(\Omega _{T})
\end{equation*}%
and%
\begin{equation*}
u^{\varepsilon }\left( x,t\right) \rightharpoonup u\left( x,t\right) \text{
in }L^{2}(0,T;H_{0}^{1}(\Omega ))\text{.}
\end{equation*}

To determine the independencies and make the local problems more precise, we
need to identify which values of $d_{i}$ and $\rho _{i}$ to use. We recall
that $d_{i}$ is the number of temporal scales faster than the square of the
spatial scale in question and $\rho _{i}$ indicates whether there is
resonance or not. Let us start with the slowest spatial scale, i.e. $i=1$.
To find $d_{1}$ we investigate on the basis of (I) how the first spatial
scale is related to the temporal scales present in the problem. We have 
\begin{equation*}
\lim_{\varepsilon \rightarrow 0}\frac{e^{\varepsilon }-1}{\left( 2\sqrt{%
\varepsilon }\right) ^{2}}=\frac{1}{4}>0
\end{equation*}%
and 
\begin{equation*}
\lim_{\varepsilon \rightarrow 0}\frac{\ln (1+\varepsilon ^{2})}{(2\sqrt{%
\varepsilon })^{2}}=0\text{,}
\end{equation*}%
which means that $d_{1}=2$. For the scale in question we have resonance
since 
\begin{equation*}
\lim_{\varepsilon \rightarrow 0}\frac{\left( 2\sqrt{\varepsilon }\right) ^{2}%
}{e^{\varepsilon }-1}=4\text{,}
\end{equation*}%
i.e., $\rho _{1}=4$ according to (II). For $i=2$ we obtain%
\begin{equation*}
\lim_{\varepsilon \rightarrow 0}\frac{\varepsilon ^{3}\ln \left( 1+\frac{1}{%
\varepsilon }\right) }{\left( \varepsilon ^{2}\right) ^{2}}=\infty
\end{equation*}%
and hence $d_{2}=0$. We also observe that 
\begin{equation*}
\lim_{\varepsilon \rightarrow 0}\frac{\left( \varepsilon ^{2}\right) ^{2}}{%
\varepsilon ^{3}\ln \left( 1+\frac{1}{\varepsilon }\right) }=0\text{,}
\end{equation*}%
which means that $\rho _{2}=0$.

Now from Theorem \ref{first inline} we have%
\begin{equation*}
\nabla u^{\varepsilon }\left( x,t\right) \overset{3,4}{\rightharpoonup }%
\nabla u\left( x,t\right) +\nabla _{y_{1}}u_{1}\left( x,t,y_{1},s_{1}\right)
+\nabla _{y_{2}}u_{2}\left( x,t,y^{2},s^{3}\right) \text{,}
\end{equation*}%
where $u\in W_{2}^{1}(0,T;H_{0}^{1}(\Omega ),L^{2}(\Omega ))$, $u_{1}\in
L^{2}(\Omega _{T};W_{2_{{\large \sharp }}}^{1}(S_{1};H_{\sharp }^{1}(Y_{1})/%
\mathbb{R}
,L_{\sharp }^{2}(Y_{1})/%
\mathbb{R}
))$ and $u_{2}\in L^{2}(\Omega _{T}\times \mathcal{Y}_{1,3};H_{\sharp
}^{1}(Y_{2})/%
\mathbb{R}
)$. Here $u$ is the unique solution to%
\begin{eqnarray*}
\partial _{t}u\left( x,t\right) -\nabla \cdot b\left( x,t,\nabla u\left(
x,t\right) \right) &=&f\left( x,t\right) \text{ in }\Omega _{T}\text{,} \\
u\left( x,t\right) &=&0\text{ on }\partial \Omega \times (0,T)\text{,} \\
u\left( x,0\right) &=&u^{0}\left( x\right) \text{ in }\Omega
\end{eqnarray*}%
with 
\begin{gather*}
b\left( x,t,\nabla u\left( x,t\right) \right) \\
=\int_{\mathcal{Y}_{2,3}}a\left( y^{2},s^{3},\nabla u\left( x,t\right)
+\nabla _{y_{1}}u_{1}\left( x,t,y_{1},s_{1}\right) +\nabla
_{y_{2}}u_{2}\left( x,t,y^{2},s^{3}\right) \right) dy^{2}ds^{3}
\end{gather*}%
and we have the two local problems%
\begin{gather*}
4\partial _{s_{1}}u_{1}\left( x,t,y_{1},s_{1}\right) -\nabla _{y_{1}}\cdot
\int_{S_{2}}\int_{S_{3}}\int_{Y_{2}}a\left( y^{2},s^{3},\nabla u\left(
x,t\right) \right. \\
+\left. \nabla _{y_{1}}u_{1}\left( x,t,y_{1},s_{1}\right) +\nabla
_{y_{2}}u_{2}\left( x,t,y^{2},s^{3}\right) \right) dy_{2}ds_{3}ds_{2}=0
\end{gather*}%
and%
\begin{equation*}
-\nabla _{y_{2}}\cdot a\left( y^{2},s^{3},\nabla u\left( x,t\right) +\nabla
_{y_{1}}u_{1}\left( x,t,y_{1},s_{1}\right) +\nabla _{y_{2}}u_{2}\left(
x,t,y^{2},s^{3}\right) \right) =0\text{.}
\end{equation*}

\end{document}